\documentclass{article}

\usepackage{arxiv}
\usepackage[utf8]{inputenc} 
\usepackage[T1]{fontenc}    
\usepackage{hyperref}       
\usepackage{url}            
\usepackage{booktabs}       
\usepackage{amsfonts}       
\usepackage{nicefrac}       
\usepackage{microtype}      
\usepackage{lipsum}
\usepackage{amsmath,amssymb}
\usepackage{amsthm}
\usepackage{graphicx}
\usepackage{epstopdf}
\usepackage{listings}
\usepackage{color}
\usepackage{fancyhdr}
\usepackage[framemethod=tikz]{mdframed}
\usepackage{subfig}
\usepackage[nottoc,notlot,notlof]{tocbibind}
\usepackage{csvsimple}
\usepackage{longtable}
\usepackage{colortbl}
\usepackage{float}
\usepackage[]{algorithm2e}
\usepackage{verbatim}
\usepackage{lscape}
\usepackage{multirow}
\usepackage{hhline}

\newcommand{\R}{\mathbb{R}}
\newcommand{\N}{\mathbb{N}}
\newcommand{\Z}{\mathbb{Z}}

\mdfdefinestyle{thm}{
	nobreak=true,
	hidealllines=true,
	leftline=true,
	innerleftmargin=10pt,
	innerrightmargin=10pt,
	innertopmargin=0pt,
}

\surroundwithmdframed[style=thm]{definition}
\newtheorem{proposition}{Proposition}[section]\surroundwithmdframed[style=thm]{proposition}
\newtheorem{theorem}{Theorem}[section]\surroundwithmdframed[style=thm]{theorem}
\newtheorem{lemma}{Lemma}[section]\surroundwithmdframed[style=thm]{lemma}
\surroundwithmdframed[style=thm]{property}
\newtheorem{problem}{Problem}[section]\surroundwithmdframed[style=thm]{problem}
\newtheorem*{problem*}{Problem}\surroundwithmdframed[style=thm]{problem*}

\theoremstyle{remark}

\newtheorem*{note}{\textbf{Remark}}

\title{Adjoint Method in PDE-based Image Compression}

\author{
  Zakaria BELHACHMI \\ 
  IRIMAS, University of Haute-Alsace, France \\
  \texttt{zakaria.belhachmi@uha.fr} \\
   \And
  Thomas JACUMIN \\
  LTH, University of Lund, Sweden \\
  \texttt{thomas.jacumin@math.lth.se} \\
  IRIMAS, University of Haute-Alsace, France \\
  \texttt{thomas.jacumin@uha.fr} \\
}

\begin{document}

\maketitle

\begin{abstract}

We consider a shape optimization  based method for finding the best interpolation data in the compression of images with noise. The aim is to reconstruct missing regions by means of minimizing a data fitting term in an $L^p$-norm between original images and their reconstructed counterparts using linear diffusion PDE-based inpainting. Reformulating the problem as a constrained optimization over sets (shapes), we derive the topological asymptotic expansion of the considered shape functionals with respect to the insertion of small ball (a single pixel) using the adjoint method. Based on the achieved distributed topological shape derivatives, we propose a numerical approach to determine the optimal set and present numerical experiments showing the efficiency of our method. Numerical computations are presented that confirm the usefulness of our theoretical findings for PDE-based image compression.
\end{abstract}


\section*{Introduction and Related Works}

PDE-based methods have attracted growing interest by researchers and engineers in image analysis field 
during the last decades \cite{Perona1990, Catte1992, Weickert1996, Masnou1998, Weickert1999, Bertalmio2000, Scherzer2008, Lenzen2011, Larnier2012, Adam2017}.
Actually, such methods have reached their maturity both from the point of view of modeling and scientific computing allowing them to be used in modern image technologies and their various applications. Image compression is one of the domain where they appear among the state-of-the-art methods \cite{Chan2001, Galic2008, Schmaltz2009, Bae2010, Peter2015, Andris2016, Mohideen2020}. In fact, the aim for such problems is to store few pixels of a given image (coding phase) and to recover/restore  
the missing part in an accurate way (decoding). The PDE-based methods use a diffusion differential operator for the inpainting of missed parts 
from an available data (boundary or small parts of the initial image) therefore their efficiency for decoding is guaranteed/encoded in the operator without any pre- or post-treatment. The question then is how to ensure with these methods a good choice, if it exists, of the ``best'' pixels to store for high quality reconstruction of the entire image?
An answer to this question is given in \cite{Belhachmi2009, Belhachmi2022Jun} for the harmonic or the heat equation, where its reformulation as a constrained (shape) optimisation problem permitted to exhibit an optimal set of pixels to do the job. In addition, analytic selection criteria using topological asymptotics were derived. Due to the simple structure of the shape functionals considered in these previous works, the topological expansion is easily derived (more or less with formal computations) and gives an analytic criterion to characterize the optimal set in compression. 
The limitation in obtaining the topological expansion this way is twofold : the criterion gives pointwise information on the importance of the location (pixel) to store which results in hard thresholding selection strategy not robust with respect to the noise. Second, the technique is limited to simple functionals, namely an $L^2$ data-fitting term and a linear diffusion operator.

The main contribution of this article is the use of the adjoint method \cite{Amstutz2006}-\cite{Garreau2001} to derive a soft analytic criterion for PDE-based compression. Though we restrict ourselves to second order linear inpainting, the method applies without significant changes to more general elliptic operators and as we show in the article to several types of noise. 

In fact, we consider the compression problem in the same framework than \cite{Belhachmi2009}, but we introduce a new approach to the characterization of the set of pixels to select using the adjoint method \cite{Guillaume2002, Garreau2001, Amstutz2006, Belhachmi2018}. This approach to obtain the topological expansion is more general than the one previously studied for the same problem in \cite{Belhachmi2009}, in the sense that it may be used for other diffusion operators and nonlinear data-fitting term, moreover it allows a better stability with respect to noise. In particular, when the accuracy of the reconstruction (fidelity term) is measured with an $L^p$-norm, $p>1$ and $p\not=2$, the adjoint method is still linear and no significant complexity or cost are added. Thus, the main results in the article include the rigorous derivation of the topological expansion based on the adjoint problem in the spirit of \cite{Amstutz2006}-\cite{Garreau2001}. 
We notice that the Dirichlet boundary condition in the inclusion prevents from a direct transposition of the method based on a local perturbation of the material properties by inserting small holes. Therefore, we adapt the sensibility analysis to the problem under consideration and we perform the asymptotic expansion of the proposed shape functional using non-standard perturbation techniques combined with truncation techniques. The asymptotic allows us to deduce a gradient
algorithm for the reconstruction that we implement and compare to previous works \cite{Belhachmi2009, Belhachmi2022Jun}.

The article is organized as follows : in Section \ref{sec:problem-formulation}, we introduce the compression problem that takes the form of a constrained optimization problem of finding the best set of pixels to store, denoted $K$. Section \ref{sec:topological-derivatives} is devoted to describe the adjoint method to compute the topological derivative of the cost functional considered. In Section \ref{sec:summary-numerical}, we perform the computations to obtain the topological expansion and the ``shape'' derivatives which involve the direct and adjoint states. Finally, in Section \ref{sec:numerical-results}, we describe the resulting algorithm and we give some numerical results to confirm the usefulness of the theory. Some of the technical proofs and auxiliary estimates are given in appendices for ease of readability. 


\section{Problem Formulation}
\label{sec:problem-formulation}

Let $D\subset\R^2$ and $f: D\to\R^d$, $d\geq 1$ a given image in some region $K\subset\subset D$. We consider the mixed elliptic boundary problem for a  given $u_0$ in $L^2(D)$, \\

\begin{problem} Find $u$ in $H^1(D)$ such that
	\begin{equation}
		\left\{\begin{array}{rl}
			u - \alpha \Delta u = u_0, & \text{in}\ D\setminus K, \\
			u = f, & \text{in}\ K, \\
			\frac{\partial u}{\partial \mathbf{n}} = 0, & \text{on}\ \partial D. \\
		\end{array}\right .\label{eq:problem_1}
	\end{equation}
	\label{pb:problem_1}
\end{problem}
where the available data $f$ is a Dirichlet ``boundary'' condition and with homogeneous Neumann boundary condition on $\partial D$. This PDE corresponds to the first term in the time discretization of the homogeneous heat equation, where we assume that the initial condition is $u_0$. For compatibility condition with the ``boundary'' data on $K$, we take as $u_0$ the image $f\in H^1(D)$, with $\Delta f\in L^2(D)$ and such that $\frac{\partial f}{\partial n}=0$ on $\partial D$.
In the compression step (coding phase), the datum $f$ is available in the entire domain $D$, so we can set the initial condition $u_0$ to the function $f$. The result of this coding step consists of a set $K$ of the pixels to store and the values of $f$ on $K$. In the decompression step (decoding phase), the data is only available in the subset $K$ of the domain, so we set $u_0=0$ (at least in $D\setminus K$). When the reconstruction is performed by solving the heat equation, it means that we start with an initial datum which do not satisfy the compatibility conditions, but 
this does not influence the dynamic as far as the convergence to an equilibrium, that is a steady state, holds (regularizing effect).
Setting $v=u-f$, we can write equivalently \\

\begin{problem} Find $v$ in $H^1(D)$ such that
	\begin{equation}
		\left\{\begin{array}{rl}
			- \alpha \Delta v + v = \alpha \Delta f, & \text{in}\ D\setminus K, \\
			v = 0, & \text{in}\ K, \\
			\frac{\partial v}{\partial \mathbf{n}} = 0, & \text{on}\ \partial D. \\
		\end{array}\right .\label{eq:problem_1:v}
	\end{equation}
	\label{pb:problem_1:v}
\end{problem}
Denoting by $v_K=u_K-f$ the solution of Problem \ref{pb:problem_1:v}, the question is to identify the region $K$ which gives the “best” approximation $u_K$, in a suitable sense, that is to say which minimizes some $L^p$-norm. The constrained optimization problem for the compression reads \cite{Belhachmi2009}, for $p>1$,
\begin{equation}\label{eq1}
	\min_{K\subseteq D,\ m(K)\leq c}\Big\{  \frac{1}{p}\int_D |u_K-f|^p\ dx\ \Big|\ u_K\ \text{solution of Problem}\ \ref{pb:problem_1} \Big\},
\end{equation}
where $m$ is a ``size measure". The optimization problem \eqref{eq1} is studied in \cite{Belhachmi2009} and the existence of an optimal set is established for $p=2$ and $m$ is the capacity of sets \cite{Ziemer1969} and the result extends to $p>1$ as noticed in \cite{Belhachmi2022Jun}. The optimal set $K$ is obtained via a relaxation procedure but its regularity is not considered yet, nevertheless the relaxation technique allows us to derive first order optimality conditions via topological derivatives, which was done in \cite{Belhachmi2009} in the case of the Laplacian as inpainting operator. In this article, we aim to compute the topological gradient \cite{Cea1973, Amstutz2006} of the shape functional using the adjoint method which possesses two main advantages on the previous approaches: it is more general and systematic with respect to the inpainting operator and the exponent $p\geq 1$, on one side and secondly, it leads to a better characterization of the relevant pixels as it gives a distribution of such pixels taking into account local information from their neighborhood. Loosely speaking, to obtain the topological derivative, let $x_0\in D$ and $K_\varepsilon=K\cup \overline{B(x_0,\varepsilon)}$ ($B(x_0,\epsilon)$ denotes the ball centred at $x_0$ with radius $\varepsilon$), then we look for an expansion of the form
\[ J(u_{K_\varepsilon}) - J(u_K) = \rho(\varepsilon)\,G(x_0) + o\big( \rho(\varepsilon) \big). \]
where $\rho$ is a positive function going to zero with $\varepsilon$ and $G$ is the so called topological
gradient \cite{Amstutz2006, Belhachmi2018, Garreau2001}. Therefore, to minimize the cost functional $J$, one has to create small holes
at the locations $x$ where $G(x)$ is the most negative. For the compression problem this amounts to select the locations where the pixels are the most important to keep.

\section{Adjoint Method and variations of the cost functional}
\label{sec:topological-derivatives}

The adjoint method has been extensively studied and succefully applied to a large number of second order elliptic problems and Helmholtz equation (see \cite{Amstutz2006,Garreau2001} and references therin). We will recall the main principle (theorem) of the method and apply it to our specific setting. 
We introduce the following abstract result which describes the adjoint method for the computation of the first variation of a given cost functional (see for instance \cite{Amstutz2006}). Let $V$ be a Hilbert space. For $\varepsilon\in [0,\zeta]$, $\zeta> 0$, we consider a symmetric bilinear form $a_\varepsilon: V\times V\to \R$ and a linear form $l_\varepsilon: V\to \R$ such that the following assumptions are fulfilled 
\begin{itemize}
    \item $|a_\varepsilon(v,w)| \leq M_1\|v\|\|w\|$, $\forall (v,w)\in V\times V$ (\textbf{continuity of the bilinear form}),
    \item $a_\varepsilon(v,v) \geq \xi\|v\|^2$, $\forall v\in V$ (\textbf{uniform coercivity}),
    \item $|l_\varepsilon(w)| \leq M_2\|w\|$, $\forall w\in V$ (\textbf{continuity of the linear form}),
\end{itemize}
with $\alpha,M_1,M_2>0$ independent of $\varepsilon$. Moreover, we suppose that there exists a continuous bilinear form $\delta a: V\times V\to \R$, a continuous linear form $\delta l: V\to \R$ and a function $\rho: \R_+\to \R_+$ such that, for all $\varepsilon\geq 0$,
\begin{itemize}
    \item $\|a_\varepsilon-a_0-\rho(\varepsilon)\,\delta a\|_{\mathcal{L}_2(V)} = o\big(\rho(\varepsilon)\big)$,
    \item $\|l_\varepsilon-l_0-\rho(\varepsilon)\, \delta l\|_{\mathcal{L}(V)} = o\big(\rho(\varepsilon)\big)$,
    \item $\lim_{\varepsilon\to 0} \rho(\varepsilon) = 0$.
\end{itemize}
We emphasize that $\delta a$ and $\delta l$ do not depend on $\varepsilon$. Finally, for all $\varepsilon\in [0,\zeta]$, consider a functional $J_\varepsilon:V\to \R$, Fr\'echet-differentiable at the point $v_0$. Assume further that there exists a number $\delta J(v_0)$ such that
\[ J_\varepsilon(w) - J_0(v) = DJ_0(v)(w-v) + \rho(\varepsilon)\,\delta J(v) + o\big(\|w-v\|+\rho(\varepsilon)\big),\ \forall (v,w)\in V\times V. \]
Then we have \cite{Amstutz2006} \\

\begin{theorem}\label{thm1} Let $v_\varepsilon \in V$ be the solution of the following problem : find $v\in V$ such that,
    \[ a_\varepsilon(v,\varphi) = l_\varepsilon(\varphi),\ \forall \varphi\in V. \]
    Let $w_0$ be the solution of the so-called adjoint problem : find $w\in V$ such that
    \[ a_0(w,\varphi) = - DJ_0(v_0)\varphi, \ \forall\varphi\in V. \]
    Then,
    \[ J_\varepsilon(v_\varepsilon) - J_0(v_0) = \rho(\varepsilon)\big(\delta a(v_0,w_0) - \delta l(w_0) + \delta J(v_0)\big) + o\big(\rho(\varepsilon)\big). \]
    \label{thm:adjoint-method}
\end{theorem}

To be more specific, for $x_0\in D$ and $r>0$, we denote by $B_r$ the open ball centred at $x_0$ and of radius $r$.
We set 
$$V_\varepsilon :=  \{ v\in H^1(D\setminus B_\varepsilon)\ |\ v=0\ \text{on}\ \partial B_\varepsilon \}.$$
Then we consider the boundary value problem : \\
\begin{problem} Find $\widetilde{v}_\varepsilon$ in $V_\varepsilon$ such that
	\begin{equation}
    	\left \{ \begin{array}{cl}
            -\alpha\Delta \widetilde{v}_\varepsilon + \widetilde{v}_\varepsilon = h, & \text{in}\ D\setminus B_\varepsilon, \\
            \widetilde{v}_\varepsilon = 0, & \text{in}\ B_\varepsilon, \\
            \partial_n \widetilde{v}_\varepsilon = 0, & \text{on}\ \partial D.
        \end{array}
        \right .\label{eq:problem:u-D}
	\end{equation}
	\label{pb:problem:u-D}
\end{problem}
with $h:= \alpha\Delta f$, but $h$ can be any $L^2(D)$ function. We denote $\widetilde{v}_0$ the solution of the problem \\
\begin{problem} Find $\widetilde{v}_0$ in $H^1(D)$ such that
	\begin{equation}
		\left \{ \begin{array}{cl}
            -\alpha\Delta \widetilde{v}_0 + \widetilde{v}_0 = h, & \text{in}\ D, \\
            \partial_n \widetilde{v}_0 = 0, & \text{on}\ \partial D.
        \end{array}
        \right .\label{eq:problem:v-D}
	\end{equation}
	\label{pb:problem:v-D}
\end{problem}
The weak formulation of problems above reads, find $\widetilde{v}_\varepsilon$ in $V_\varepsilon$ such that, for all $\varphi$ in $V_\varepsilon\cap H^1(D)$, we have
\[ \tilde{a}_\varepsilon(\widetilde{v}_\varepsilon, \varphi) = \tilde{l}_\varepsilon(\varphi), \]
with,
\begin{align*}
    & \tilde{a}_\varepsilon(\widetilde{v}_\varepsilon, \varphi) := \alpha\int_{D\setminus B_\varepsilon} \nabla \widetilde{v}_\varepsilon\cdot\nabla\varphi\ dx + \int_{D\setminus B_\varepsilon} \widetilde{v}_\varepsilon\,\varphi\ dx, \\
    & \tilde{l}_\varepsilon(\varphi) := \int_{D\setminus B_\varepsilon} h\,\varphi\ dx.
\end{align*}
The dependency of the space $V_\varepsilon$ on $\varepsilon$ prevents us from using Theorem \ref{thm1} directly, therefore, we introduce a  truncation technique \cite{Guillaume2002}, which consists of 
inserting a ball $B_R$, for a fixed $R>\varepsilon$ and splitting Problem \ref{pb:problem:u-D} into two sub-problems that we glue at their common boundary (see Figure \ref{fig:domain-split}). More precisely, we consider the sub-problems : an \textit{internal} problem
\[ \left \{ \begin{array}{cl}
    -\alpha\Delta v_{\varepsilon,R} + v_{\varepsilon,R} = h, & \text{in}\ B_R\setminus B_\varepsilon, \\
    v_{\varepsilon,R} = 0, & \text{on}\ \partial B_\varepsilon, \\
    v_{\varepsilon,R} = v_\varepsilon, & \text{on}\ \partial B_R,
\end{array} \right .\]
and an \textit{external} problem
\[ \left \{ \begin{array}{cl}
    -\alpha\Delta v_\varepsilon + v_\varepsilon = h, & \text{in}\ D\setminus B_R, \\
    \partial_n v_\varepsilon = \partial_n v_{\varepsilon,R}, & \text{on}\ \partial B_R, \\
    \partial_n v_\varepsilon = 0, & \text{on}\ \partial D.
\end{array}
\right .\]

\begin{figure}[H]
    \centering
    \subfloat[Before splitting.]{
        \begin{tikzpicture}[scale=0.35]
            
            \draw[fill=black!10] 
                (-6.75, 15) .. controls + (-1,-0.5) and + (1,1.5) ..
                (-12.5, 10.75) .. controls + (-1,-1.5) and + (-0.5,1.5) ..
                (-13.25, 0) .. controls + (0.5,-1.5) and + (-3,-1) ..
                (-5, 4.75) .. controls + (3,1) and + (1,-2) ..
                (5, 7.25) .. controls + (-1,2) and + (2,-1) ..
                (-2.75, 14.75) .. controls + (-2,1) and + (1,0.5) ..
                (-6.75, 15) -- cycle;
                
            \node (6) at (-11,3.5) {$D\setminus B_\varepsilon$};
                
        	\node[circle, draw=black, fill=white, inner sep=0pt,minimum size=25pt] (1) at (-5,10) {$B_\varepsilon$};
        	
        	\draw[->] (-13,14) -- (-10,10);
        	
        	\node (3) at (-13.5,14.5) {$\widetilde{v}_\varepsilon$};
        \end{tikzpicture}
    }
	\qquad
	\subfloat[After splitting.]{
        \begin{tikzpicture}[scale=0.35]
            
            \draw[fill=black!10] 
                (-6.75, 15) .. controls + (-1,-0.5) and + (1,1.5) ..
                (-12.5, 10.75) .. controls + (-1,-1.5) and + (-0.5,1.5) ..
                (-13.25, 0) .. controls + (0.5,-1.5) and + (-3,-1) ..
                (-5, 4.75) .. controls + (3,1) and + (1,-2) ..
                (5, 7.25) .. controls + (-1,2) and + (2,-1) ..
                (-2.75, 14.75) .. controls + (-2,1) and + (1,0.5) ..
                (-6.75, 15) -- cycle;
            
            \node (6) at (-11,3.5) {$D\setminus B_R$};
                
            \node[circle, dashed, draw=black, fill=black!25, inner sep=0pt,minimum size=70pt] (1) at (-5,10) {};
        	\node (6) at (-5,7.8) {$B_R\setminus B_\varepsilon$};
        	\node[circle, draw=black, fill=white, inner sep=0pt,minimum size=25pt] (2) at (-5,10) {$B_\varepsilon$};
        	
        	\draw[->] (-13,14) -- (-10,10);
        	\node (3) at (-13.5,14.5) {$v_\varepsilon$};
        	\draw[->] (1,14) -- (-2.8,11);
        	\node (3) at (1.5,14.5) {$v_{\varepsilon,R}$};
        \end{tikzpicture}
    }
    \caption{Illustration of the splitting.}
	\label{fig:domain-split}
\end{figure}


As the two sub-problems transform the initial one into a transmission problem. We have \\

\begin{proposition} We have,
    \[ \widetilde{v}_\varepsilon = \begin{cases} 
        v_\varepsilon, & \text{in}\ D\setminus B_R, \\
        v_{\varepsilon,R}, & \text{in}\ B_R\setminus B_\varepsilon.
    \end{cases} \]
    \label{prop:gluing-solutions}
\end{proposition}
\begin{proof}
    We set \begin{equation} v := \begin{cases} 
        v_\varepsilon & ,\ \text{in}\ D\setminus B_R, \\
        v_{\varepsilon,R} & ,\ \text{in}\ B_R\setminus B_\varepsilon.
    \end{cases} \label{eq:gluing-solutions}\end{equation}
    Let $\varphi$ be in $V_\varepsilon$, then,
    \begin{align*}
        \tilde{a}_\varepsilon(v, \varphi) & = \alpha\int_{D\setminus B_R} \nabla v\cdot\nabla\varphi\ dx + \int_{D\setminus B_R} v\,\varphi\ dx + \alpha\int_{B_R\setminus B_\varepsilon} \nabla v\cdot\nabla\varphi\ dx + \int_{B_R\setminus B_\varepsilon} v\,\varphi\ dx \\
        %
        %
    \end{align*}
    Replacing $v$ by its expression \eqref{eq:gluing-solutions} and integrating by parts yields,
    \begin{align*}
        %
        \tilde{a}_\varepsilon(v, \varphi) & = \int_{D\setminus B_R} (-\alpha\Delta v_\varepsilon + v_\varepsilon)\,\varphi\ dx + \int_{B_R\setminus B_\varepsilon} (-\alpha\Delta v_{\varepsilon,R} + v_{\varepsilon,R})\,\varphi\ dx \\ & \hspace{2cm} + \alpha\int_{\partial B_R} \partial_{n_\text{int}} v_\varepsilon\,\varphi\ d\sigma + \alpha\int_{\partial B_R} \partial_{n_\text{ext}} v_{\varepsilon,R}\,\varphi\ d\sigma \\
        %
        %
        %
        & = \int_{D\setminus B_\varepsilon} h\,\varphi\ dx = \tilde{l}_\varepsilon(\varphi).
    \end{align*}
    By the uniqueness of the solution of Problem \ref{pb:problem:u-D}, we have $v=\widetilde{v}_\varepsilon$.
\end{proof}

For the \textit{internal} problem, we introduce the notation $v_\varepsilon^{h,\phi}$ instead of $v_{\epsilon, R}$, the solution of the more general problem \\

\begin{problem} Find $v_\varepsilon^{h,\phi}$ in $\{ v\in H^1(B_R\setminus B_\varepsilon)\ |\ v=0\ \text{on}\ \partial B_\varepsilon \}$ such that
	\begin{equation}
    	\left \{ \begin{array}{cl}
            -\alpha\Delta v_\varepsilon^{h,\phi} + v_\varepsilon^{h,\phi} = h, & \text{in}\ B_R\setminus B_\varepsilon, \\
            v_\varepsilon^{h,\phi} = 0, & \text{on}\ \partial B_\varepsilon, \\
            v_\varepsilon^{h,\phi} = \phi, & \text{on}\ \partial B_R.
        \end{array}
        \right .\label{eq:problem:v-epsilon-inside}
	\end{equation}
	\label{pb:problem:v-epsilon-inside}
\end{problem}
Therefore, $v_{\varepsilon,R} = v_\varepsilon^{h,\phi}$, when $\phi=v_\varepsilon$. We also notice that,
\[ v_\varepsilon^{h,\phi} = v_\varepsilon^{h,0} + v_\varepsilon^{0,\phi}. \]
We remind the Dirichlet-to-Neumann operator $T_\varepsilon : H^{1/2}(\partial B_R) \to H^{-1/2}(\partial B_R)$ by 
\[ T_\varepsilon(\phi) := \nabla v_\varepsilon^{0,\phi}\cdot n. \]
and we set \[ h_\varepsilon := -\nabla v_\varepsilon^{h,0}\cdot n\in H^{-1/2}(\partial B_R). \]
Hence, setting $V_R=H^1(D\setminus B_R)$, we can rewrite the \textit{external} problem using this operator as following (we still denote by $v_\varepsilon$ the solution) : \\

\begin{problem} Find $v_\varepsilon$ in $V_R$ such that
	\begin{equation}
    	\left \{ \begin{array}{cl}
            -\alpha\Delta v_\varepsilon + v_\varepsilon = h, & \text{in}\ D\setminus B_R, \\
            -\partial_n v_\varepsilon + T_\varepsilon v_\varepsilon = h_\varepsilon, & \text{on}\ \partial B_R, \\
            \partial_n v_\varepsilon = 0, & \text{on}\ \partial D.
        \end{array}
        \right .\label{eq:problem:v-epsilon-outside}
	\end{equation}
	\label{pb:problem:v-epsilon-outside}
\end{problem}

For $\varepsilon\in [0,\zeta]$, $R>\zeta$, and $v,\varphi$ in $V_R := H^1(D\setminus B_R)$, we define
\begin{align*}
    & a_\varepsilon(v,\varphi) := \alpha\int_{D\setminus B_R}\nabla v\cdot\nabla \varphi\ dx + \alpha\int_{\partial B_R} T_\varepsilon v\,\varphi\ d\sigma + \int_{D\setminus B_R}v\,\varphi\ dx, \\
    &  l_\varepsilon(\varphi) := \int_{D\setminus B_R} h\,\varphi\ dx + \alpha\int_{\partial B_R} h_\varepsilon\, \varphi\ d\sigma.
\end{align*}
So that the associated variational formulation reads : find $v\in V_R$, such that  
\[ a_\varepsilon(v,\varphi) = l_\varepsilon(\varphi),\qquad\forall \varphi\in V_R. \]
It is easily checked that $a_\varepsilon$ is symmetric and $l_\varepsilon$ is continuous.
%
%

We take as cost function, for $p>1$,
\[ \widetilde{J}_\varepsilon(\widetilde{v}) := \int_{D\setminus B_\varepsilon} |h-\widetilde{v}|^p\ dx,\ \forall \widetilde{v}\in V_\varepsilon, \]

We define now the cost functional on $V_R$ as follows : for $v\in V_R$, we set $\widetilde{v}_\varepsilon \in V_\varepsilon$ the extension of $v$ in $D\setminus B_\varepsilon$ such that,
\begin{itemize}
    \item[\textbullet] $\widetilde{v}_\varepsilon|_{D\setminus B_R} = v$,
    \item[\textbullet] $\widetilde{v}_\varepsilon|_{B_R\setminus B_\varepsilon} = v_\varepsilon^{h,\phi}$, for $\phi=v$ on $\partial B_R$.
\end{itemize}
We notice that $\widetilde{v}_\varepsilon$ do not satisfy Problem \ref{pb:problem:u-D} except if $v$ is the solution of Problem \ref{pb:problem:v-epsilon-outside}. Then, we may define the restriction of $\widetilde{J}_\varepsilon$ to $V_R$ by : 
\[ J_\varepsilon (v) := \widetilde{J}_\varepsilon(\widetilde{v}_\varepsilon),\ \forall v\in V_R. \]
\subsection{The Adjoint Problem and Related Estimates}
\label{sec:adjoint}

We state now the adjoint problem associated to Problem \ref{pb:problem:v-epsilon-outside} when $\varepsilon=0$ : we denote by $w_0$ the weak solution in $V_R$ of \[ a_0(w_0,\varphi) = -DJ_0(v_0)\,\varphi,\ \forall\varphi\in V_R, \]
where $v_0$ is the solution of Problem \ref{pb:problem:v-epsilon-outside}. The adjoint state $w_0$ is then the solution of \\
\begin{problem} Find $w_0$ in $V_R$ such that
	\begin{equation}
    	\left \{ \begin{array}{cl}
            -\alpha\Delta w_0 + w_0 = -v_0|v_0|^{p-2}, & \text{in}\ D\setminus B_R, \\
            -\partial_n w_0 + T_0 w_0 = h_0, & \text{on}\ \partial B_R, \\
            \partial_n w_0 = 0, & \text{on}\ \partial D.
        \end{array}
        \right .\label{eq:problem:w-epsilon-outside}
	\end{equation}
	\label{pb:problem:w-epsilon-outside}
\end{problem}
We aim to find $\delta a$, $\delta l$ and $\delta J$ from the adjoint method, Theorem \ref{thm:adjoint-method}. Let $h\in L^2(D)$ and $\phi\in H^{1/2}(\partial B_R)$. 

\subsection{Variations of the Bilinear Form} 

We start by giving an explicit formulation for both $v_0^{0,\phi}$ and $v_\varepsilon^{0,\phi}$, which is analogous to \cite{Samet2003}, with the following proposition : \\

\begin{proposition} For $\phi$ in $H^{1/2}(\partial B_R)$, we have,
    \[ v_0^{0,\phi} (r, \theta) = \sum_{n\in\mathbb{Z}} \frac{I_n(\alpha^{-1/2}r)}{I_n(\alpha^{-1/2}R)}\,\phi_n\,e^{in\theta}, \]
    and,
    \[ v_\varepsilon^{0,\phi}(r, \theta) = \sum_{n\in\mathbb{Z}} \frac{I_n(\alpha^{-1/2}\varepsilon)K_n(\alpha^{-1/2}r) - K_n(\alpha^{-1/2}\varepsilon)I_n(\alpha^{-1/2}r)}{I_n(\alpha^{-1/2}\varepsilon)K_n(\alpha^{-1/2}R) - K_n(\alpha^{-1/2}\varepsilon)I_n(\alpha^{-1/2}R)}\,\phi_n\,e^{in\theta}, \]
    where $(r,\theta)$ are the polar coordinates in $\R^2$, $(\phi_n)_n$ are the Fourier coefficients of $\phi$, $I_n$ and $K_n$ are the modified Bessel functions of the first and second kind respectively \cite{Abramowitz1972, Oldham2009}.
\end{proposition}
\begin{proof}
    Using polar coordinates in $\R^2$, we have,
    \[ v_\varepsilon^{h,0}(r, \theta) = \sum_{n\in\Z} c_n(r) \,e^{in\theta} \qquad \text{and} \qquad h(r, \theta) = \sum_{n\in\Z} h_n(r) \,e^{in\theta}, \]
    where $c_n$ satisfies, for all $n$ in $\Z$, and $0<r\leq R$,
    \begin{equation*}
        -\alpha\,r^2\,c_n''(r) - \alpha\,r\,c_n'(r) + (r^2+\alpha n^2) c_n(r) = 0.
    \end{equation*}
    We solve the equation, and we get,
    $$ c_{0,n}(r) = A_{0,n} I_n(\alpha^{-1/2}\,r), $$
    and
    $$ c_{\varepsilon,n}(r) = A_{\varepsilon,n} I_n(\alpha^{-1/2}\,r) + B_{\varepsilon,n} K_n(\alpha^{-1/2}\,r). $$
    By using the boundaries conditions, we have the result.
\end{proof}
\begin{note}
    According to \cite{Abramowitz1972,Oldham2009}, we have for all $n$ in $\Z$, $I_{-n} = I_n$ and $K_{-n} = K_n$.
\end{note}

Next, we state the variation of the solution with respect to hole's radius : \\

\begin{proposition} For $\phi$ in $H^{1/2}(\partial B_R)$, we set,
    \[ \delta v^{0,\phi}(r) := -\phi_0\frac{I_0(\alpha^{-1/2}R)K_0(\alpha^{-1/2}r) - K_0(\alpha^{-1/2}R)I_0(\alpha^{-1/2}r)}{I_0(\alpha^{-1/2} R)^2}. \]
    Then, for $\varepsilon$ sufficiently small, we have the following asymptotic estimation,
    \[ v_\varepsilon^{0,\phi}(r, \theta) - v_0^{0,\phi}(r, \theta) - \frac{-1}{\ln{\varepsilon}}\delta v^{0,\phi}(r) = o\left(\frac{-1}{\ln{\varepsilon}}\right). \]
\end{proposition}
\begin{proof} We have,
    \[ v_\varepsilon^{0,\phi}(r, \theta) - v_0^{0,\phi}(r, \theta) = \left( \frac{I_0(\alpha^{-1/2}\varepsilon)K_0(\alpha^{-1/2}r) - K_0(\alpha^{-1/2}\varepsilon)I_0(\alpha^{-1/2}r)}{I_0(\alpha^{-1/2}\varepsilon)K_0(\alpha^{-1/2}R) - K_0(\alpha^{-1/2}\varepsilon)I_0(\alpha^{-1/2}R)} - \frac{I_0(\alpha^{-1/2}r)}{I_0(\alpha^{-1/2}R)} \right)\,\phi_0 + R_\varepsilon(r,\theta), \]
    where, \[ R_\varepsilon(r,\theta) := \sum_{n\in\mathbb{Z}^*} \left( \frac{I_n(\alpha^{-1/2}\varepsilon)K_n(\alpha^{-1/2}r) - K_n(\alpha^{-1/2}\varepsilon)I_n(\alpha^{-1/2}r)}{I_n(\alpha^{-1/2}\varepsilon)K_n(\alpha^{-1/2}R) - K_n(\alpha^{-1/2}\varepsilon)I_n(\alpha^{-1/2}R)} - \frac{I_n(\alpha^{-1/2}r)}{I_n(\alpha^{-1/2}R)} \right)\,\phi_n\,e^{in\theta}. \]
    Then,
    \[ v_\varepsilon^{0,\phi}(r, \theta) - v_0^{0,\phi}(r, \theta) = -\delta v^{0,\phi}(r)\,\frac{I_0(\alpha^{-1/2} R) I_0(\alpha^{-1/2}\varepsilon)}{I_0(\alpha^{-1/2}\varepsilon)K_0(\alpha^{-1/2}R) - K_0(\alpha^{-1/2}\varepsilon)I_0(\alpha^{-1/2}R)} + R_\varepsilon(r,\theta). \]
    We use that, \cite{Abramowitz1972},
    \[ K_0(x) = -(\gamma-\ln{2} + \ln{x}) I_0(x) + x r_1(x), \]
    and get,
    \[ v_\varepsilon^{0,\phi}(r, \theta) - v_0^{0,\phi}(r, \theta) = -\delta v^{0,\phi}(r)\,\left( M + I_0(\alpha^{-1/2}R)\ln{\varepsilon} + \varepsilon r_2(\varepsilon)\right)^{-1} + R_\varepsilon(r,\theta), \]
    where, \[ M := \frac{K_0(\alpha^{-1/2}R) + I_0(\alpha^{-1/2}R)\big(\gamma-\ln{2} + \ln{(\alpha^{-1/2})} \big)}{I_0(\alpha^{-1/2} R)}, \]
    is a constant independent of $\varepsilon$. Finally,
    \[ v_\varepsilon^{0,\phi}(r, \theta) - v_0^{0,\phi}(r, \theta) = -\delta v^{0,\phi}(r)\,\frac{-1}{\ln{\varepsilon}}\left( 1 + \frac{M}{\ln{\varepsilon}} + \varepsilon r_3(\varepsilon)\right)^{-1} + R_\varepsilon(r,\theta). \]
    
    It remains to show that $R_\varepsilon(r,\theta) = o\left(\frac{-1}{\ln{\varepsilon}}\right)$. We have,
    $$ R_\varepsilon(r,\theta) := \frac{-1}{\ln{\varepsilon}}\Theta_\varepsilon(r,\theta), $$
    where,
    $$ \Theta_\varepsilon(r,\theta) := \sum_{n\in\mathbb{Z}^*} \ln{\varepsilon} \frac{ K_n(\alpha^{-1/2}R)I_n(\alpha^{-1/2}r) - I_n(\alpha^{-1/2}R)K_n(\alpha^{-1/2}r)}{I_n(\alpha^{-1/2}\varepsilon)K_n(\alpha^{-1/2}R) - I_n(\alpha^{-1/2}R)K_n(\alpha^{-1/2}\varepsilon)}\frac{I_n(\alpha^{-1/2}\varepsilon)}{I_n(\alpha^{-1/2}R)}\,\phi_n\,e^{in\theta}. $$
    Moreover,
    $$ \big|\Theta_\varepsilon(r,\theta)\big| \leq \sum_{n\in\mathbb{Z}^*} \big|\ln{\varepsilon}\big| \Bigg|\frac{ I_n(\alpha^{-1/2}R)K_n(\alpha^{-1/2}r) - K_n(\alpha^{-1/2}R)I_n(\alpha^{-1/2}r)}{ I_n(\alpha^{-1/2}R)K_n(\alpha^{-1/2}\varepsilon) - I_n(\alpha^{-1/2}\varepsilon)K_n(\alpha^{-1/2}R)}\Bigg|\Bigg|\frac{I_n(\alpha^{-1/2}\varepsilon)}{I_n(\alpha^{-1/2}R)}\Bigg|\,\big|\phi_n\big|. $$
    Since $I_n$ is an increasing function and $K_n$ is a decreasing function, we have, for $r\in[\varepsilon,R]$,
    $$ \big|I_n(\alpha^{-1/2}R)K_n(\alpha^{-1/2}r) - K_n(\alpha^{-1/2}R)I_n(\alpha^{-1/2}r)\big| \leq \big|I_n(\alpha^{-1/2}R)K_n(\alpha^{-1/2}\varepsilon) - K_n(\alpha^{-1/2}R)I_n(\alpha^{-1/2}\varepsilon)\big|. $$
    Thus,
    $$ \big|\Theta_\varepsilon(r,\theta)\big| \leq \sum_{n\in\mathbb{Z}^*} \big|\ln{\varepsilon}\big| \Bigg|\frac{I_n(\alpha^{-1/2}\varepsilon)}{I_n(\alpha^{-1/2}R)}\Bigg|\,\big|\phi_n\big|. $$
    Finally, using that $$ I_n(\alpha^{-1/2} \varepsilon) = \frac{1}{n!}\left(\frac{\alpha^{-1/2}}{2}\right)^n \varepsilon^n + \varepsilon^n\,r(\varepsilon), $$
    we have that $\Theta_\varepsilon$ tends to $0$ when $\varepsilon$ goes to $0$.
\end{proof}
By noticing that \cite{Abramowitz1972},
\begin{equation} \label{eq:wronskian}
    I_0(x)K_0'(x) - K_0(x)I_0'(x) = -I_0(x)K_1(x) - K_0(x)I_1(x) = - W\big(K_0(x),I_0(x)\big) = - \frac{1}{x},
\end{equation}
with $W$ the Wronskian and by using the previous result, we have the following : \\

\begin{proposition} For $\phi$ in $H^{1/2}(\partial B_R)$, we define,
    \[ \delta T(\phi) := \phi_0\frac{1}{R\, I_0(\alpha^{-1/2} R)^2}. \]
    Then, for $\varepsilon$ sufficiently small, we have the following asymptotic estimation,
    \[ \left\|T_\varepsilon - T_0 - \frac{-1}{\ln{\varepsilon}}\delta T\right\|_{\mathcal{L}\big(H^{1/2}(\partial B_R), H^{-1/2}(\partial B_R)\big)} = o\left(\frac{-1}{\ln{\varepsilon}}\right). \]
\end{proposition}

Finally, we can derive from the previous proposition the variations of the bilinear form : \\

\begin{proposition} For $\phi$ in $H^{1/2}(\partial B_R)$, we define,
    \[ \delta a(v,w) := \alpha\frac{v^\text{mean}}{I_0(\alpha^{-1/2} R)}\frac{w^\text{mean}}{I_0(\alpha^{-1/2} R)}, \]
    where $v^\text{mean}$ and $w^\text{mean}$ denote the mean value of $v$ and $w$ on $\partial B_R$. Then, for $\varepsilon$ sufficiently small, we have the following asymptotic estimation,
    \[ \left|a_\varepsilon(v,w) - a_0(v,w) - \frac{-2\pi}{\ln{\varepsilon}}\delta a(v,w)\right| = o\left(\frac{-1}{\ln{\varepsilon}}\right)\,\|v\|_{V_R} \|w\|_{V_R},\ \forall v,w \in V_R. \]
\end{proposition}

\subsection{Variations of the Linear Form}

We give an explicit formulation for both $v_0^{h,0}$ and $v_\varepsilon^{h,0}$ with the following proposition : \\

\begin{proposition} \label{prop:solution-unhomo} For $h$ in $L^2(B_R)$, we have,
    \[ v_0^{h,0} (r, \theta) = \sum_{n\in\mathbb{Z}} \Big( \big( A_{0,n} +  A_{n}^\text{p}(r) \big) I_n(\alpha^{-1/2}\,r) + B_{n}^\text{p}(r) K_n(\alpha^{-1/2}\,r) \Big)\,e^{in\theta}, \]
    with,
    \begin{align*}
        & A_{n}^\text{p}(r) := -\alpha^{-1}\int_0^r s\,K_n(\alpha^{-1/2}\, s) h_n(s)\ ds, \\
        & B_{n}^\text{p}(r) := \alpha^{-1}\int_0^r s\,I_n(\alpha^{-1/2}\, s) h_n(s)\ ds, \\
        & A_{0,n} := -\frac{ A_{n}^\text{p}(R) I_n(\alpha^{-1/2}\, R) + B_{n}^\text{p}(R) K_n(\alpha^{-1/2}\, R) }{I_n(\alpha^{-1/2}\, R)}.
    \end{align*}
    
    Moreover, 
    \[ v_\varepsilon^{h,0}(r, \theta) = \sum_{n\in\mathbb{Z}} \Big(\big( A_{\varepsilon,n} +  A_{n}^\text{p}(r) \big) I_n(\alpha^{-1/2}\,r) + \big(B_{\varepsilon,n} + B_{n}^\text{p}(r)\big) K_n(\alpha^{-1/2}\,r) \Big) \,e^{in\theta}, \]
    with,
    \begin{align*}
        & A_{\varepsilon, n} = \frac{K_n(\alpha^{-1/2}\varepsilon)\big( A_{n}^\text{p}(R)I_n(\alpha^{-1/2}R) + B_{n}^\text{p}(R)K_n(\alpha^{-1/2}R) \big)}{I_n(\alpha^{-1/2}\varepsilon)K_n(\alpha^{-1/2}R) - K_n(\alpha^{-1/2}\varepsilon)I_n(\alpha^{-1/2}R)}, \\
        & B_{\varepsilon, n} = -\frac{I_n(\alpha^{-1/2}\varepsilon)\big( A_{n}^\text{p}(R)I_n(\alpha^{-1/2}R) + B_{n}^\text{p}(R)K_n(\alpha^{-1/2}R) \big)}{I_n(\alpha^{-1/2}\varepsilon)K_n(\alpha^{-1/2}R) - K_n(\alpha^{-1/2}\varepsilon)I_n(\alpha^{-1/2}R)}.
    \end{align*}
\end{proposition}
\begin{proof}
    Using polar coordinates in $\R^2$, we have,
    \[ v_\varepsilon^{h,0}(r, \theta) = \sum_{n\in\Z} c_n(r) \,e^{in\theta} \qquad \text{and} \qquad h(r, \theta) = \sum_{n\in\Z} h_n(r) \,e^{in\theta}, \]
    where $c_n$ satisfies, for all $n$ in $\Z$ and $0<r\leq R$,
    \begin{equation} \label{eq:fourier-coeff-equation}
        -\alpha\,r^2\,c_n''(r) - \alpha\,r\,c_n'(r) + (r^2+\alpha n^2) c_n(r) = r^2 h_n(r).
    \end{equation}
    Firstly, we solve the homogeneous equation, and we get,
    $$ c_{0,n}^\text{h}(r) = A_{0,n} I_n(\alpha^{-1/2}\,r), $$
    and
    $$ c_{\varepsilon,n}^\text{h}(r) = A_{\varepsilon,n} I_n(\alpha^{-1/2}\,r) + B_{\varepsilon,n} K_n(\alpha^{-1/2}\,r). $$
    Secondly, we use the variation of parameters method to get the particular solution,
    \begin{equation} \label{eq:particular-solution}
        c_n^\text{p}(r) = A_{n}^\text{p}(r) I_n(\alpha^{-1/2}\,r) + B_{n}^\text{p}(r) K_n(\alpha^{-1/2}\,r).
    \end{equation}
    By replacing \eqref{eq:particular-solution} into \eqref{eq:fourier-coeff-equation}, and by supposing that, $$ (A_{n}^\text{p})'(r) I_n(\alpha^{-1/2}\,r) + (B_{n}^\text{p})'(r) K_n(\alpha^{-1/2}\,r) = 0, $$
    we get,
    $$ (A_{n}^\text{p})'(r) I_n'(\alpha^{-1/2}\,r) + (B_{n}^\text{p})'(r) K_n'(\alpha^{-1/2}\,r) = -\alpha^{-1/2} h_n(r). $$
    Solving the last two equations, we get the value of $A_{n}^\text{p}(r)$ and $B_{n}^\text{p})(r)$ as stated in the theorem. Finally, by the superposition principle and by using the boundaries conditions, we have the result.
\end{proof}


\begin{note} 
    Since $h$ is in $L^2(D)$, using the Parseval's equality we have that the Fourier coefficients $h_n$ are in $L^2(D)$ as well. Similarly for the Fourier coefficients $c_n$. As a result, we have that the integral in the $A_{n}^\text{p}$ is convergent.
\end{note}

\begin{note} 
    We have, $$ v_0^{h,0} (x_0) = \sum_{n\in\mathbb{Z}} A_{0,n} I_n(0)\,e^{in\theta},\ \forall \theta \in [0,2\pi[, $$
    and in particular, for $\theta=0$, we have 
    $$ v_0^{h,0} (x_0) = \sum_{n\in\mathbb{Z}} A_{0,n} I_n(0). $$
    Moreover, \cite{Abramowitz1972, Oldham2009}, $I_n(0)$ vanishes for $n\in\N^*$ and $I_0(0) = 1$. Then,
    \begin{equation} \label{eq:A00}
        A_{0,0} = v_0^{h,0}(x_0).
    \end{equation}
\end{note}

Now, we can state the variation of the solution with respect to hole's radius : \\

\begin{proposition} For $h$ in $L^2(B_R)$, we set,
    \[ \delta v^{h,0}(r) := -A_{0,0} \frac{I_0(\alpha^{-1/2}\, R)K_0(\alpha^{-1/2}r) - K_0(\alpha^{-1/2}R)I_0(\alpha^{-1/2}r)}{I_0(\alpha^{-1/2}\,R)}, \]
    where $A_{0,0}$ is defined in Proposition \ref{prop:solution-unhomo}. Then, for $\varepsilon$ sufficiently small, we have the following asymptotic estimation,
    \[ v_\varepsilon^{h,0}(r, \theta) - v_0^{h,0}(r, \theta) - \frac{-1}{\ln{\varepsilon}}\delta v^{h,0}(r) = o\left(\frac{-1}{\ln{\varepsilon}}\right). \]
\end{proposition}
\begin{proof} We have,
     $$ v_\varepsilon^{h,0}(r, \theta) - v_0^{h,0}(r, \theta) = \big( A_{\varepsilon,0} - A_{0,0} \big) I_0(\alpha^{-1/2}\,r) + B_{\varepsilon,0} K_0(\alpha^{-1/2}\,r) + R_{\varepsilon}(r,\theta), $$
    and,
    $$ R_{\varepsilon}(r,\theta) := \sum_{n\in\mathbb{Z}^*} \Big( \big( A_{\varepsilon,n} - A_{0,n} \big) I_n(\alpha^{-1/2}\,r) + B_{\varepsilon,n} K_n(\alpha^{-1/2}\,r) \Big)\,e^{in\theta}. $$
    We have, for all $n$ in $\Z$,
    $$ A_{\varepsilon,n} - A_{0,n} = -A_{0,n} \frac{I_n(\alpha^{-1/2}\,\varepsilon)K_n(\alpha^{-1/2}\,R)}{I_n(\alpha^{-1/2}\,\varepsilon)K_n(\alpha^{-1/2}\,R) - K_n(\alpha^{-1/2}\,\varepsilon)I_n(\alpha^{-1/2}\,R)}, $$
    and,
    $$ B_{\varepsilon,n} = A_{0,n} \frac{I_n(\alpha^{-1/2}\varepsilon)  I_n(\alpha^{-1/2}\,R) }{I_n(\alpha^{-1/2}\varepsilon)K_n(\alpha^{-1/2}R) - K_n(\alpha^{-1/2}\varepsilon)I_n(\alpha^{-1/2}R)}. $$
    Thus, 
    $$ v_\varepsilon^{h,0}(r, \theta) - v_0^{h,0}(r, \theta) = -\delta v^{h,0}(r) \frac{ I_0(\alpha^{-1/2}\,R) I_0(\alpha^{-1/2}\,\varepsilon) }{I_0(\alpha^{-1/2}\,\varepsilon)K_0(\alpha^{-1/2}\,R) - K_0(\alpha^{-1/2}\,\varepsilon)I_0(\alpha^{-1/2}\,R)} + R_{\varepsilon}(r), $$
    and,
    $$ R_{\varepsilon}(r,\theta) := -\sum_{n\in\mathbb{Z}^*} A_{0,n} \frac{I_n(\alpha^{-1/2}\,\varepsilon)\Big( I_n(\alpha^{-1/2}\,r)K_n(\alpha^{-1/2}\,R) - K_n(\alpha^{-1/2}\,r)I_n(\alpha^{-1/2}\,R) \Big) }{I_n(\alpha^{-1/2}\,\varepsilon)K_n(\alpha^{-1/2}\,R) - K_n(\alpha^{-1/2}\,\varepsilon)I_n(\alpha^{-1/2}\,R)}\,e^{in\theta}. $$
    We use that, \cite{Abramowitz1972},
    \[ K_0(x) = -(\gamma-\ln{2} + \ln{x}) I_0(x) + x r_1(x), \]
    and get,
    $$ v_\varepsilon^{h,0}(r, \theta) - v_0^{h,0}(r, \theta) = -\delta v^{h,0}(r) \Big( M + \ln{\varepsilon} + \varepsilon r_2(\varepsilon) \Big)^{-1} + R_{\varepsilon}(r). $$
    where, \[ M := \frac{K_0(\alpha^{-1/2}R) + I_0(\alpha^{-1/2}R)\big(\gamma-\ln{2} + \ln{(\alpha^{-1/2})} \big)}{I_0(\alpha^{-1/2} R)}, \]
    is a constant independent of $\epsilon$. It remains to show that $R_\varepsilon(r,\theta) = o\left(\frac{-1}{\ln{\varepsilon}}\right)$. We have,
    $$ R_\varepsilon(r,\theta) := \frac{-1}{\ln{\varepsilon}}\Theta_\varepsilon(r,\theta), $$
    where,
    $$ \Theta_\varepsilon(r,\theta) := \sum_{n\in\mathbb{Z}^*} \ln{\varepsilon} \frac{ K_n(\alpha^{-1/2}R)I_n(\alpha^{-1/2}r) - I_n(\alpha^{-1/2}R)K_n(\alpha^{-1/2}r)}{I_n(\alpha^{-1/2}\varepsilon)K_n(\alpha^{-1/2}R) - I_n(\alpha^{-1/2}R)K_n(\alpha^{-1/2}\varepsilon)}I_n(\alpha^{-1/2}\varepsilon)\,A_{0,n}\,e^{in\theta}. $$
    Moreover,
    $$ \big|\Theta_\varepsilon(r,\theta)\big| \leq \sum_{n\in\mathbb{Z}^*} \big|\ln{\varepsilon}\big| \Bigg|\frac{ I_n(\alpha^{-1/2}R)K_n(\alpha^{-1/2}r) - K_n(\alpha^{-1/2}R)I_n(\alpha^{-1/2}r)}{ I_n(\alpha^{-1/2}R)K_n(\alpha^{-1/2}\varepsilon) - I_n(\alpha^{-1/2}\varepsilon)K_n(\alpha^{-1/2}R)}\Bigg|\big|I_n(\alpha^{-1/2}\varepsilon)\big|\,\big|A_{0,n}\big|. $$
    Since $I_n$ is an increasing function and $K_n$ is a decreasing function, we have, for $r\in[\varepsilon,R]$,
    $$ \big|I_n(\alpha^{-1/2}R)K_n(\alpha^{-1/2}r) - K_n(\alpha^{-1/2}R)I_n(\alpha^{-1/2}r)\big| \leq \big|I_n(\alpha^{-1/2}R)K_n(\alpha^{-1/2}\varepsilon) - K_n(\alpha^{-1/2}R)I_n(\alpha^{-1/2}\varepsilon)\big|. $$
    Thus,
    $$ \big|\Theta_\varepsilon(r,\theta)\big| \leq \sum_{n\in\mathbb{Z}^*} \big|\ln{\varepsilon}\big| \big|I_n(\alpha^{-1/2}\varepsilon)\big|\,\big|A_{0,n}\big|. $$
    Finally, using that, $$ I_n(\alpha^{-1/2} \varepsilon) = \frac{1}{n!}\left(\frac{\alpha^{-1/2}}{2}\right)^n \varepsilon^n + \varepsilon^n\,r(\varepsilon), $$
    we have that $\Theta_\varepsilon$ tends to $0$ when $\varepsilon$ goes to $0$.
\end{proof}


Using the previous result and the property of the Wronskian \cite{Abramowitz1972}, we have the following : \\

\begin{proposition} For $h$ in $L^2(B_R)$, we define,
    \[ \delta h(h) := -A_{0,0}\frac{1}{R\,I_0(\alpha^{-1/2}\,R)}. \]
    where $A_{0,0}$ is defined in Proposition \ref{prop:solution-unhomo}. Then, for $\varepsilon$ sufficiently small, we have the following asymptotic estimation,
    \[  \left\|h_\varepsilon - h_0 - \frac{-1}{\ln{\varepsilon}}\delta h\right\|_{-1/2,\partial B_R} = o\left(\frac{-1}{\ln{\varepsilon}}\right). \]
\end{proposition}

Finally, we can derive from the previous proposition the variations of the linear form : \\

\begin{proposition} For $h$ in $L^2(B_R)$, we define,
    \[ \delta l(w) := -\alpha A_{0,0}\frac{w^\text{mean}}{I_0(\alpha^{-1/2} R)}. \]
    Then, for $\varepsilon$ sufficiently small, we have the following asymptotic estimation,
    \[ \left|l_\varepsilon(w) - l_0(w) - \frac{-2\pi}{\ln{\varepsilon}}\delta l(w)\right| = o\left(\frac{-1}{\ln{\varepsilon}}\right)\,\|w\|_{V_R},\ \forall w \in V_R. \]
\end{proposition}

\section{Computation of the topological derivative}
\label{sec:summary-numerical}

We now gather the previous section results to derive the topological derivative. We consider the adjoint problem of Problem \ref{pb:problem:v-D} : \\
\begin{problem} Find $\widetilde{w}_0$ in $H^1(D)$ such that
	\begin{equation}
		\left \{ \begin{array}{cl}
            -\alpha\Delta \widetilde{w}_0 + \widetilde{w}_0 = -\widetilde{v}_0|\widetilde{v}_0|^{p-2}, & \text{in}\ D, \\
            \partial_n \widetilde{w}_0 = 0, & \text{on}\ \partial D. \\
        \end{array} \right .\label{eq:problem:w-D}
	\end{equation}
	\label{pb:problem:w-D}
\end{problem}
\color{black}

Then, we have the following proposition, \\

\begin{proposition}
    $w_0$, solution of Problem \ref{pb:problem:w-epsilon-outside}, is the restriction of $\widetilde{w}_0$ to $D\setminus B_R$.
    \label{prop:w0-wD}
\end{proposition}
\begin{proof}
    We set $w_R := \widetilde{w}_0|_{D\setminus B_R}$. We have to show that $w_R=w_0$ i.e. $a_0(w_R,\varphi_R) = -DJ_0(v_0)\varphi_R$, $\forall \varphi_R\in V_R$. Let $\varphi_R\in V_R$. We denote $\widetilde{\varphi}\in V_0$ the extension of $\varphi_R$ to $V_0$ such that $-\alpha\Delta\widetilde{\varphi} + \widetilde{\varphi} = 0$ in $B_R$. Thus,
    \begin{align*}
        a_0(w_R,\varphi_R) & = \alpha\int_{D\setminus B_R}\nabla w_R\cdot\nabla\varphi_R\ dx + \alpha\int_{\partial B_R} T_0 w_R\, \varphi_R\ d\sigma + \int_{D\setminus B_R} w_R\,\varphi_R\ dx \\
        & = \alpha\int_{D\setminus B_R}\nabla w_R\cdot\nabla\varphi_R\ dx + \alpha\int_{\partial B_R} T_0 w_R\, \varphi_R\ d\sigma + \int_{D\setminus B_R} w_R\,\varphi_R\ dx + \int_{B_R}\underbrace{(-\alpha\Delta\widetilde{\varphi} + \widetilde{\varphi})}_{=0} \widetilde{w}_0\ dx. \\
        %
        %
        %
        %
    \end{align*}
    And after integrating by parts,
    \begin{align*}
        a_0(w_R,\varphi_R) & = \alpha\int_{D\setminus B_R}\nabla w_R\cdot\nabla\varphi_R\ dx + \int_{D\setminus B_R} w_R\varphi_R\ dx + \alpha\int_{B_R}\nabla\widetilde{\varphi}\cdot\nabla \widetilde{w}_0\ dx + \int_{B_R}\widetilde{\varphi}\,\widetilde{w}_0\ dx \\
        & = \alpha\int_{D}\nabla \widetilde{w}_0\cdot\nabla\widetilde{\varphi}\ dx + \int_{D}\widetilde{w}_0\,\widetilde{\varphi}\ dx \\
        & = \widetilde{a}_0(\widetilde{w}_0,\widetilde{\varphi}) = -D\widetilde{J}_0(v_D)\widetilde{\varphi}.
    \end{align*}
    Moreover, by definition $\widetilde{J}_0(v_D) = J_0(v_R)$, 
    thus, 
    \[ D\widetilde{J}_0(v_D)\widetilde{\varphi} = DJ_0(v_0)\varphi_R. \]
    By uniqueness of the solution, $w_R = w_0$.
\end{proof}

Using that, if $u$ and $v$ are solutions of the linear diffusion equation on $B_R$, $$ v(x_0) = \frac{v^\text{mean}}{I_0(\alpha^{-1/2} R)} \qquad \text{and} \qquad w(x_0) = \frac{w^\text{mean}}{I_0(\alpha^{-1/2} R)}, $$
it follows the topological gradient based on the adjoint method given by: \\

\begin{proposition} For $\varepsilon$ small enough, we have,
    \[ j(K_\varepsilon) - j(K) = 2\alpha\,\widetilde{v}_0(x_0)\,\widetilde{w}_0(x_0)\frac{-2\pi}{\ln{\varepsilon}} + o\left(\frac{-1}{\ln{\varepsilon}}\right), \]
    with $\widetilde{v}_0$ solution of Problem \ref{pb:problem:v-D} and $\widetilde{w}_0$ solution of Problem \ref{pb:problem:w-D}.
\end{proposition}

We notice that with this expansion, we get the main theoretical result of the paper which might be summarized as follows: to minimize the $L^p$-error between an image and its reconstruction from linear diffusion inpainting, we have to keep in the mask the pixels $x_0$ which minimize the product $\widetilde{v}_0(x_0)\,\widetilde{w}_0(x_0)$. Such an analytic result gives a soft threshold criterion for the selection of $K$. In fact, the adjoint state is a obtained by solving a linear PDE which is in our setting and, by the elliptic regularity $H^2$ (the right hand side is $L^2$ even when $p=1$ with the regularization adopted in the article), we obtain a continuous solution.
The distribution of this function measures the influence of a considered single pixel and its neighborhood in the cost variations. This is what we call soft-threshold (opposite to hard threshold with pixel-wise asymptotic). 
Moreover, choosing the best pixels this way may be also enhanced by halftoning techniques \cite{Ulichney1987} which increases the quality of the set selection. 
The implementation of the algorithms is done in Python and can be found at \cite{JacuminDec2023}.


    
    





\section{Numerical Results}
\label{sec:numerical-results}

In this section we present some  numerical results when the cost functional is the $L^1$-error and the $L^2$-error, respectively, as they are the most representative for noise in practice. In fact, we take these specific values of $p$, depending on the nature of the noise considered.
The reconstruction step is performed by solving the heat equation with the semi-implicit discrete scheme and $\alpha$ is the time step.
Let us denote by $f$ the original image, $f_\delta$ the noisy one, and we denote by $u$ the reconstructed image. We emphasis that the inpainting masks are built from $f_\delta$, that is to say $f_\delta$ is available in $D$ during the mask selection step while the data for the reconstruction are only available in $K$. 
We denote by \textit{Lp-ADJ-T} the algorithm using the adjoint method by selecting the pixels given by $-\widetilde{v}_0\,\widetilde{w}_0$ and we denote \textit{Lp-ADJ-H} the algortihm combining with a halftoning technique (see \cite{Ulichney1987, Floyd1976})).
For comparisons purpose, we consider \textit{H1-T} and \textit{H1-H} which correspond to the mask selection following the asymptotic expansion given in \cite{Belhachmi2009}, and where we take directly as criterion the hard/soft-thresholding of $|\Delta f_\delta|$ (see. \cite{Belhachmi2009}-\cite{Belhachmi2022Jun}).

\subsection{Salt and Pepper Noise}

A common way to deal with impulse noise like salt and pepper, is to minimize the $L^1$-error \cite{Nikolova2002, Nikolova2004}. In our algorithm, we use $p=1.01$. We give in Table \ref{tab:methods-comparison:l1:0.05}, Table \ref{tab:methods-comparison:l1:0.1} and Table \ref{tab:methods-comparison:l1:0.15} the $L^1$-error for the methods described above and several amounts of salt and/or pepper noise.

We notice that \textit{L1-ADJ-H} gives the lower $L^1$-error. In fact, the impulse noises induce a high laplacian at the location of the corrupted pixels, thus satisfy the criterion for these methods. On the other hand, the adjoint state $\widetilde{w}_0$ and $\widetilde{v}_0$ are respectively solutions to a linear PDE (i.e. $A(z)=-\alpha \Delta z+z$), with $-\vert \widetilde{v}_0\vert$ and $\Delta f$ as second members, so that they give smooth distribution (e.g. formally $\widetilde{v}_0=A^{-1}(\Delta f)$). 
In addition, and for the same reason, in the \textit{L1-ADJ-} masks, we can distinguish the edges of the image, while its not the case with the \textit{H1-} methods, so that the asymptotic given by the adjoint method is more edge-preserving. Interestingly, the \textit{L1-ADJ-H} method gives also better visual results than the \textit{H1-H} method when the image is free from any noise.

In Figure \ref{fig:sp:0-0:l1:0.1}, Figure \ref{fig:sp:0.02-0:l1:0.1}, Figure \ref{fig:sp:0-0.02:l1:0.1} and Figure \ref{fig:sp:0.01-0.01:l1:0.1}, the resulting masks and reconstruction are given for different level of noise. We observe that most of the corrupted pixels are not selected in $K$ for the \textit{L1-ADJ-} methods, while they are selected in the case of \textit{H1-}ones.

\begin{table}[H]
    \centering
    \begin{tabular}{|c|c||c|c||c|c||c||c|}
        \hline
        \multicolumn{2}{|c||}{\textbf{Noise}} & \multicolumn{2}{c||}{\textbf{L1-ADJ-T}} & \multicolumn{2}{c||}{\textbf{L1-ADJ-H}} & \textbf{H1-T} & \textbf{H1-H} \\
        \hline
        Salt & Pepper & $\alpha$ & $\|f-u\|_1$ & $\alpha$ & $\|f-u\|_1$ & $\|f-u\|_1$ & $\|f-u\|_1$ \\
        \hhline{|========|}
        0 & 0 & 0.01 & 7232.60 & 3.62 & \textbf{1631.97} & 8100.57 & 1961.39 \\
        0.02 & 0 & 2.67 & 4390.30 & 1.01 & \textbf{1936.35} & 13128.15 & 11086.31 \\
        0 & 0.02 & 1.47 & 3886.39 & 0.96 & \textbf{1819.94} & 15656.18 & 12950.13 \\
        0.01 & 0.01 & 0.46 & 3783.87 & 0.76 & \textbf{2006.60} & 7086.75 & 6591.59 \\
        0.04 & 0 & 2.67 & 5355.38 & 0.56 & \textbf{2327.53} & 25753.78 & 20149.18 \\
        0 & 0.04 & 1.61 & 4395.32 & 1.72 & \textbf{2176.40} & 30912.89 & 24934.66 \\
        0.02 & 0.02 & 0.41 & 4782.36 & 0.71 & \textbf{3035.47} & 13596.56 & 12835.65 \\
        0.1 & 0 & 3.02 & 9694.11 & 0.56 & \textbf{2937.17} & 30139.92 & 27912.42 \\
        0 & 0.1 & 5.38 & 6336.82 & 0.56 & \textbf{2757.27} & 35384.75 & 33442.85 \\
        0.05 & 0.05 & 0.36 & 7756.14 & 0.41 & \textbf{5639.58} & 37021.45 & 28992.83 \\
        \hline
    \end{tabular} \vspace{0.2cm}
    
    \caption{$L^1$-error between the original image $f$ and the reconstruction $u$ (built from  $f_\delta$)  with $5\%$ of total pixels saved.}
    \label{tab:methods-comparison:l1:0.05}
\end{table}

\begin{table}[H]
    \centering
    \begin{tabular}{|c|c||c|c||c|c||c||c|}
        \hline
        \multicolumn{2}{|c||}{\textbf{Noise}} & \multicolumn{2}{c||}{\textbf{L1-ADJ-T}} & \multicolumn{2}{c||}{\textbf{L1-ADJ-H}} & \textbf{H1-T} & \textbf{H1-H} \\
        \hline
        Salt & Pepper & $\alpha$ & $\|f-u\|_1$ & $\alpha$ & $\|f-u\|_1$ & $\|f-u\|_1$ & $\|f-u\|_1$ \\
        \hhline{|========|}
        0 & 0 & 0.01 & 4241.43 & 2.27 & \textbf{918.08} & 4648.99 & 961.01 \\
        0.02 & 0 & 0.41 & 3195.78 & 0.66 & \textbf{1220.02} & 6531.56 & 3875.63 \\
        0 & 0.02 & 0.36 & 2496.19 & 0.71 & \textbf{1218.38} & 5461.22 & 4019.73 \\
        0.01 & 0.01 & 0.56 & 2241.62 & 0.76 & \textbf{1318.16} & 4573.62 & 3097.18 \\
        0.04 & 0 & 0.36 & 4073.13 & 0.56 & \textbf{1435.63} & 12381.64 & 9611.86 \\
        0 & 0.04 & 2.07 & 2947.03 & 0.56 & \textbf{1444.24} & 15171.93 & 11195.98 \\
        0.02 & 0.02 & 0.46 & 2963.49 & 0.61 & \textbf{2057.61} & 6654.28 & 5901.86 \\
        0.1 & 0 & 0.26 & 8332.90 & 0.51 & \textbf{1807.20} & 28163.84 & 22580.17 \\
        0 & 0.1 & 2.42 & 4318.63 & 0.51 & \textbf{1852.08} & 34035.09 & 27595.06 \\
        0.05 & 0.05 & 0.36 & 6176.51 & 0.51 & \textbf{4796.29} & 17716.39 & 15806.10 \\
        \hline
    \end{tabular} \vspace{0.2cm}
    
    \caption{$L^1$-error between the original image $f$ and the reconstruction $u$ (built from  $f_\delta$)  with $10\%$ of total pixels saved.}
    \label{tab:methods-comparison:l1:0.1}
\end{table}

\begin{table}[H]
    \centering
    \begin{tabular}{|c|c||c|c||c|c||c||c|}
        \hline
        \multicolumn{2}{|c||}{\textbf{Noise}} & \multicolumn{2}{c||}{\textbf{L1-ADJ-T}} & \multicolumn{2}{c||}{\textbf{L1-ADJ-H}} & \textbf{H1-T} & \textbf{H1-H} \\
        \hline
        Salt & Pepper & $\alpha$ & $\|f-u\|_1$ & $\alpha$ & $\|f-u\|_1$ & $\|f-u\|_1$ & $\|f-u\|_1$ \\
        \hhline{|========|}
        0 & 0 & 0.01 & 2004.71 & 1.16 & 629.31 & 3069.92 & \textbf{629.18} \\
        0.02 & 0 & 0.46 & 1999.43 & 0.61 & \textbf{910.40} & 3681.41 & 2169.59 \\
        0 & 0.02 & 0.41 & 1612.94 & 0.71 & \textbf{940.36} & 3453.97 & 2263.76 \\
        0.01 & 0.01 & 0.61 & 1645.00 & 0.61 & \textbf{1002.72} & 2936.01 & 2057.86 \\
        0.04 & 0 & 0.41 & 2637.04 & 0.61 & \textbf{1142.34} & 7601.93 & 5195.02 \\
        0 & 0.04 & 0.36 & 2221.72 & 0.71 & \textbf{1148.26} & 7299.16 & 5652.56 \\
        0.02 & 0.02 & 0.51 & 2197.91 & 0.56 & \textbf{1656.26} & 5063.83 & 4101.42 \\
        0.1 & 0 & 0.31 & 5985.97 & 0.56 & \textbf{1459.34} & 21270.53 & 16461.51 \\
        0 & 0.1 & 0.36 & 4068.65 & 0.51 & \textbf{1497.94} & 26678.66 & 20173.49 \\
        0.05 & 0.05 & 0.41 & 4999.00 & 0.51 & \textbf{4098.81} & 11612.23 & 10171.57 \\
        \hline
    \end{tabular} \vspace{0.2cm}
    
    \caption{$L^1$-error between the original image $f$ and the reconstruction $u$ (built from  $f_\delta$)  with $15\%$ of total pixels saved.}
    \label{tab:methods-comparison:l1:0.15}
\end{table}

\begin{figure}[H]
	\centering
	\subfloat[Without noise.]{
		\includegraphics[width=3.6cm]{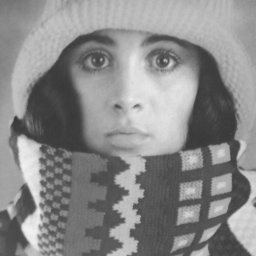}
	}
	\quad
	\subfloat[With $2\%$ of salt noise.]{
		\includegraphics[width=3.6cm]{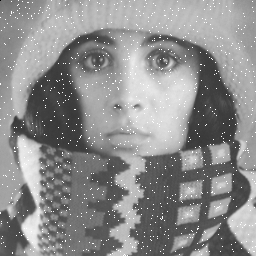}
	}
	\quad
	\subfloat[With $2\%$ of pepper noise.]{
		\includegraphics[width=3.6cm]{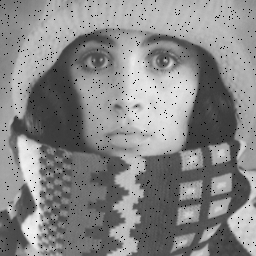}
	}
	\quad
	\subfloat[With $2\%$ of salt and pepper noise.]{
		\includegraphics[width=3.6cm]{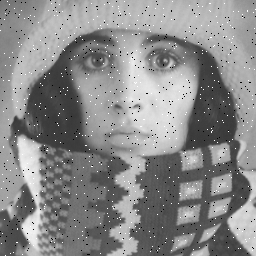}
	}
	\caption{Input images.}
\end{figure}

\begin{figure}[H]
	\centering
	\subfloat[Mask L1-ADJ-T.]{
		\includegraphics[width=3.6cm]{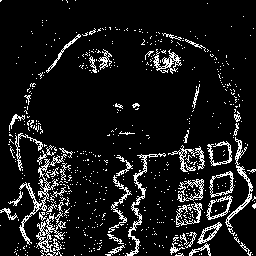}
	}
	\quad
	\subfloat[Reconstruction with L1-ADJ-T.]{
		\includegraphics[width=3.6cm]{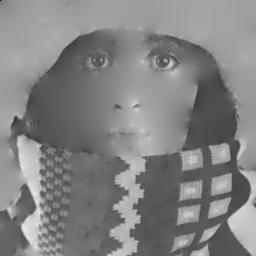}
	}
	\quad
	\subfloat[Mask L1-ADJ-H.]{
		\includegraphics[width=3.6cm]{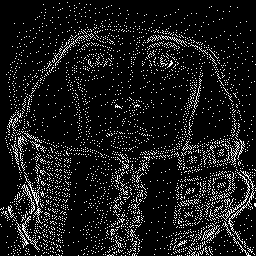}
	}
	\quad
	\subfloat[Reconstruction with L1-ADJ-H.]{
		\includegraphics[width=3.6cm]{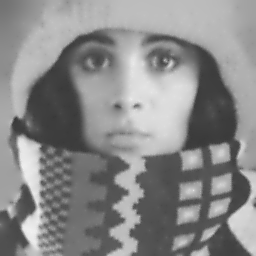}
	}
\end{figure}
\begin{figure}[H]
	\centering
	\subfloat[Mask H1-T.]{
		\includegraphics[width=3.6cm]{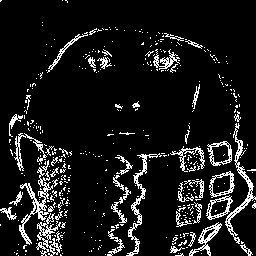}
	}
	\quad
	\subfloat[Reconstruction with H1-T.]{
		\includegraphics[width=3.6cm]{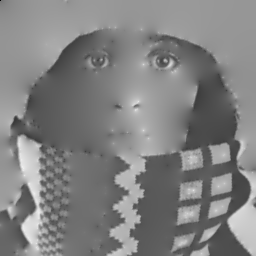}
	}
	\quad
	\subfloat[Mask H1-H.]{
		\includegraphics[width=3.6cm]{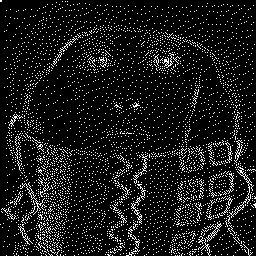}
	}
	\quad
	\subfloat[Reconstruction with H1-H.]{
		\includegraphics[width=3.6cm]{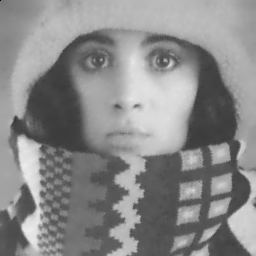}
	}
	\caption{Masks and reconstructions from image without noise and with $10\%$ of total pixels saved.}
	\label{fig:sp:0-0:l1:0.1}
\end{figure}

\begin{figure}[H]
	\centering
	\subfloat[Mask L1-ADJ-T.]{
		\includegraphics[width=3.6cm]{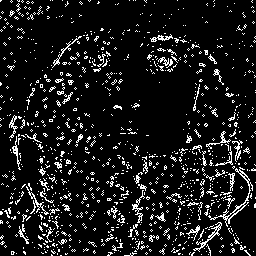}
	}
	\quad
	\subfloat[Reconstruction with L1-ADJ-T.]{
		\includegraphics[width=3.6cm]{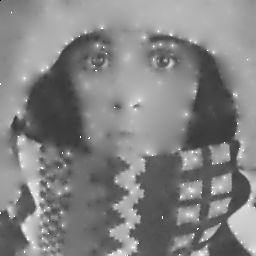}
	}
	\quad
	\subfloat[Mask L1-ADJ-H.]{
		\includegraphics[width=3.6cm]{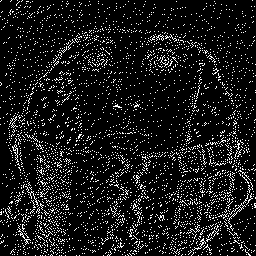}
	}
	\quad
	\subfloat[Reconstruction with L1-ADJ-H.]{
		\includegraphics[width=3.6cm]{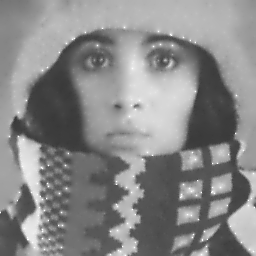}
	}
\end{figure}
\begin{figure}[H]
	\centering
	\subfloat[Mask H1-T.]{
		\includegraphics[width=3.6cm]{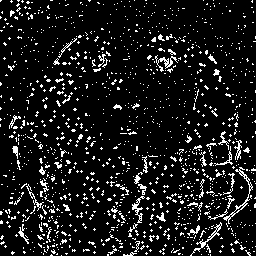}
	}
	\quad
	\subfloat[Reconstruction with H1-T.]{
		\includegraphics[width=3.6cm]{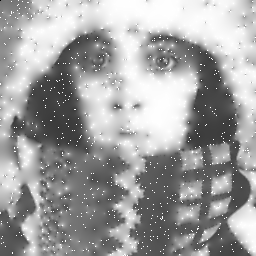}
	}
	\quad
	\subfloat[Mask H1-H.]{
		\includegraphics[width=3.6cm]{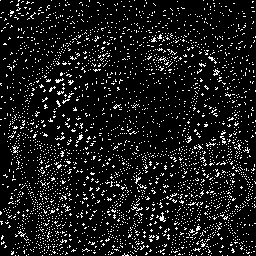}
	}
	\quad
	\subfloat[Reconstruction with H1-H.]{
		\includegraphics[width=3.6cm]{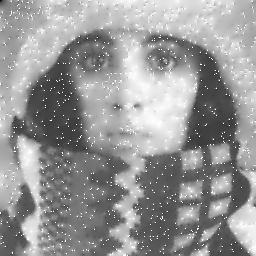}
	}
	\caption{Masks and reconstructions from image with $2\%$ of salt noise and with $10\%$ of total pixels saved.}
	\label{fig:sp:0.02-0:l1:0.1}
\end{figure}

\begin{figure}[H]
	\centering
	\subfloat[Mask L1-ADJ-T.]{
		\includegraphics[width=3.6cm]{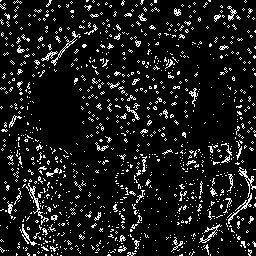}
	}
	\quad
	\subfloat[Reconstruction with L1-ADJ-T.]{
		\includegraphics[width=3.6cm]{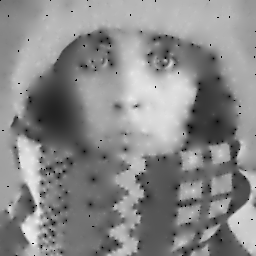}
	}
	\quad
	\subfloat[Mask L1-ADJ-H.]{
		\includegraphics[width=3.6cm]{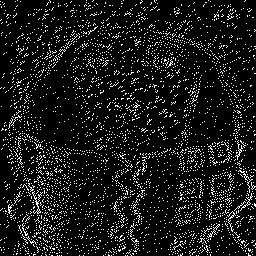}
	}
	\quad
	\subfloat[Reconstruction with L1-ADJ-H.]{
		\includegraphics[width=3.6cm]{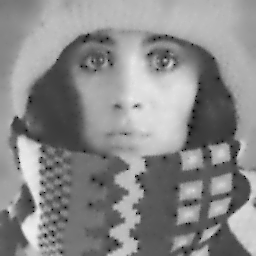}
	}
\end{figure}
\begin{figure}[H]
	\centering
	\subfloat[Mask H1-T.]{
		\includegraphics[width=3.6cm]{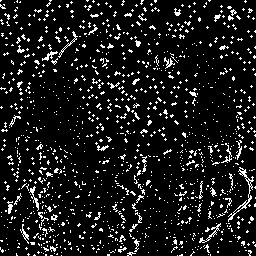}
	}
	\quad
	\subfloat[Reconstruction with H1-T.]{
		\includegraphics[width=3.6cm]{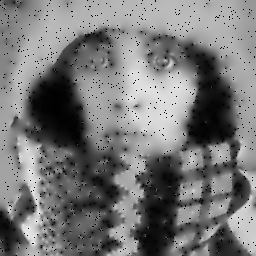}
	}
	\quad
	\subfloat[Mask H1-H.]{
		\includegraphics[width=3.6cm]{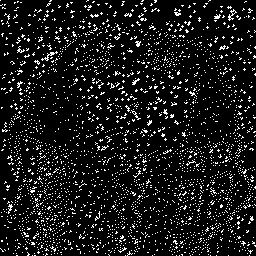}
	}
	\quad
	\subfloat[Reconstruction with H1-H.]{
		\includegraphics[width=3.6cm]{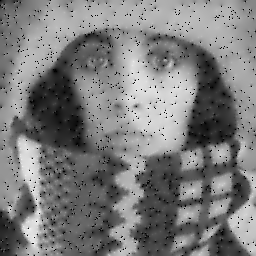}
	}
	\caption{Masks and reconstructions from image with $2\%$ of pepper noise and with $10\%$ of total pixels saved.}
	\label{fig:sp:0-0.02:l1:0.1}
\end{figure}

\begin{figure}[H]
	\centering
	\subfloat[Mask L1-ADJ-T.]{
		\includegraphics[width=3.6cm]{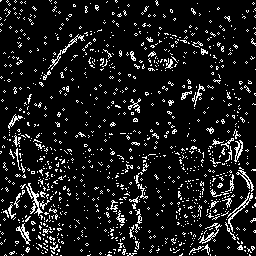}
	}
	\quad
	\subfloat[Reconstruction with L1-ADJ-T.]{
		\includegraphics[width=3.6cm]{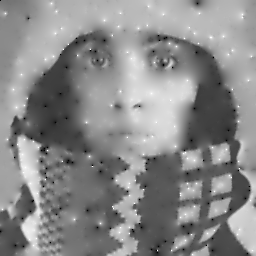}
	}
	\quad
	\subfloat[Mask L1-ADJ-H.]{
		\includegraphics[width=3.6cm]{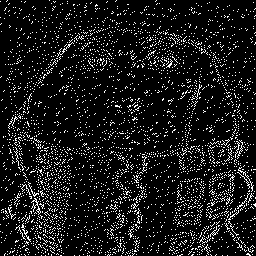}
	}
	\quad
	\subfloat[Reconstruction with L1-ADJ-H.]{
		\includegraphics[width=3.6cm]{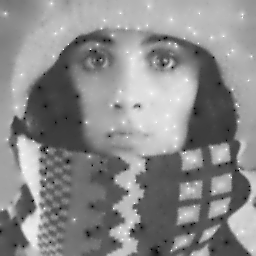}
	}
\end{figure}
\begin{figure}[H]
	\centering
	\subfloat[Mask H1-T.]{
		\includegraphics[width=3.6cm]{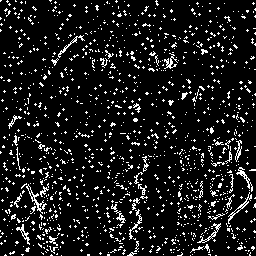}
	}
	\quad
	\subfloat[Reconstruction with H1-T.]{
		\includegraphics[width=3.6cm]{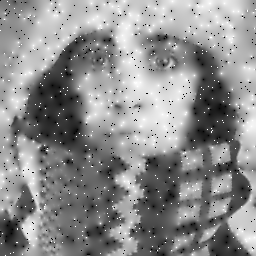}
	}
	\quad
	\subfloat[Mask H1-H.]{
		\includegraphics[width=3.6cm]{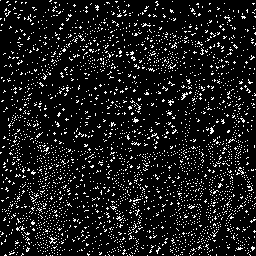}
	}
	\quad
	\subfloat[Reconstruction with H1-H.]{
		\includegraphics[width=3.6cm]{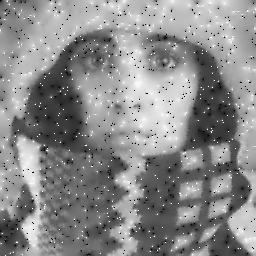}
	}
	\caption{Masks and reconstructions from image with $2\%$ of salt and pepper noise and with $10\%$ of total pixels saved.}
	\label{fig:sp:0.01-0.01:l1:0.1}
\end{figure}

\subsection{Gaussian Noise}

Now, we consider images with gaussian noise. In this case we take $p=2$, and although the algorithms which are not based on the adjoint method perform well, we notice that this method gives better results again.
We give in Table \ref{tab:methods-comparison:l2:0.05}, Table \ref{tab:methods-comparison:l2:0.1} and Table \ref{tab:methods-comparison:l2:0.15} the $L^2$-error for the methods \textit{L2-ADJ-T}, \textit{L2-ADJ-H}, \textit{H1-T} and \textit{H1-H} with respect to the deviation $\sigma>0$ of gaussian noise. 
Formally, the criterion $-\widetilde{v}_0\,\widetilde{w}_0$ is close to $|\Delta f|^2$ which is similar to the result found in \cite{Belhachmi2022Jun}, but the fact that the adjoint state $w$ and primal variable $v$ are computed by solving linear PDEs improves distribution of the topological derivative. We see that for a reasonable level of noise, the \textit{L2-ADJ-H} gives lower $L^2$-error and that the reconstructed image seems to have less noise than the original one. Similarly to \textit{L1-ADJ-}methods, we can distinguish the edges of the image in the \textit{L2-ADJ-}masks, while its not the case with the \textit{H1-}methods.

We plot in Figure \ref{fig:w:0:l2:0.1}, Figure \ref{fig:w:0.03:l2:0.1} and Figure \ref{fig:w:0.05:l2:0.1} the resulting masks and reconstruction from various noise level.

\begin{table}[H]
    \centering
    \begin{tabular}{|c||c|c||c|c||c||c|}
        \hline
        \textbf{Noise} & \multicolumn{2}{c||}{\textbf{L2-ADJ-T}} & \multicolumn{2}{c||}{\textbf{L2-ADJ-H}} & \textbf{H1-T} & \textbf{H1-H} \\
        \hline
        $\sigma$ & $\alpha$ & $\|f-u\|_2$ & $\alpha$ & $\|f-u\|_2$ & $\|f-u\|_2$ & $\|f-u\|_2$ \\
        \hhline{|=======|}
        0    & 0.01 & 35.02 & 2.62 & 16.98 & 39.17 & \textbf{9.78} \\
        0.03 & 0.31 & 13.88 & 1.37 & \textbf{9.53} & 13.76 & 12.33 \\
        0.05 & 0.66 & 15.18 & 2.07 & \textbf{12.51} & 17.01 & 15.49 \\
        0.1 & 1.16 & 30.48 & 1.81 & \textbf{23.57} & 31.48 & 23.99 \\
        0.2 & 0.01 & 67.51 & 0.01 & 52.99 & 77.29 & \textbf{42.16} \\
        \hline
    \end{tabular} \vspace{0.2cm}
    
    \caption{$L^2$-error between the original image $f$ and the reconstruction $u$ (built from  $f_\delta$)  with $5\%$ of total pixels saved.}
    \label{tab:methods-comparison:l2:0.05}
\end{table}

\begin{table}[H]
    \centering
    \begin{tabular}{|c||c|c||c|c||c||c|}
        \hline
        \textbf{Noise} & \multicolumn{2}{c||}{\textbf{L2-ADJ-T}} & \multicolumn{2}{c||}{\textbf{L2-ADJ-H}} & \textbf{H1-T} & \textbf{H1-H} \\
        \hline
        $\sigma$ & $\alpha$ & $\|f-u\|_2$ & $\alpha$ & $\|f-u\|_2$ & $\|f-u\|_2$ & $\|f-u\|_2$ \\
        \hhline{|=======|}
        0    & 0.01 & 23.08 & 0.01 & 9.70 & 25.57 & \textbf{4.99} \\
        0.03 & 0.71 &  9.38 & 0.96 & \textbf{7.91} & 9.23 & 8.57 \\
        0.05 & 0.86 & 13.25 & 0.76 & \textbf{12.39} & 13.78 & 12.55 \\
        0.1  & 0.71 & 26.95 & 0.66 & 24.47 & 27.15 & \textbf{22.95} \\
        0.2  & 0.01 & 56.08 & 2.27 & 46.46 & 61.94 & \textbf{42.62} \\
        \hline
    \end{tabular} \vspace{0.2cm}
    
    \caption{$L^2$-error between the original image $f$ and the reconstruction $u$ (built from $f_\delta$)  with $10\%$ of total pixels saved.}
    \label{tab:methods-comparison:l2:0.1}
\end{table}

\begin{table}[H]
    \centering
    \begin{tabular}{|c||c|c||c|c||c||c|}
        \hline
        \textbf{Noise} & \multicolumn{2}{c||}{\textbf{L2-ADJ-T}} & \multicolumn{2}{c||}{\textbf{L2-ADJ-H}} & \textbf{H1-T} & \textbf{H1-H} \\
        \hline
        $\sigma$ & $\alpha$ & $\|f-u\|_2$ & $\alpha$ & $\|f-u\|_2$ & $\|f-u\|_2$ & $\|f-u\|_2$ \\
        \hhline{|=======|}
        0    & 0.01 & 11.62 & 0.01 & 6.58 & 18.21 & \textbf{3.35} \\
        0.03 & 0.71 &  8.25 & 0.56 & \textbf{7.69} & 8.14 & 7.64 \\
        0.05 & 0.51 & 12.79 & 0.66 & \textbf{12.23} & 13.00 & 11.91 \\
        0.1  & 0.31 & 25.95 & 0.76 & 24.32 & 25.89 & \textbf{22.98} \\
        0.2  & 0.01 & 51.20 & 1.11 & 46.93 & 54.71 & \textbf{43.29}  \\
        \hline
    \end{tabular} \vspace{0.2cm}
    
    \caption{$L^2$-error between the original image $f$ and the reconstruction $u$ (built from  $f_\delta$)  with $15\%$ of total pixels saved.}
    \label{tab:methods-comparison:l2:0.15}
\end{table}

\begin{figure}[H]
	\centering
	\subfloat[Without noise.]{
		\includegraphics[width=3.6cm]{images/Trui.png}
	}
	\quad
	\subfloat[$\sigma=0.03$.]{
		\includegraphics[width=3.6cm]{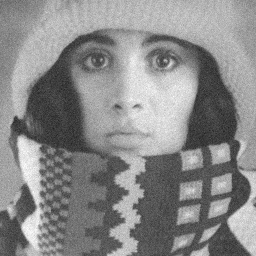}
	}
	\quad
	\subfloat[$\sigma=0.05$.]{
		\includegraphics[width=3.6cm]{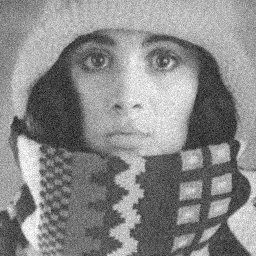}
	}
	\quad
	\subfloat[$\sigma=0.1$.]{
		\includegraphics[width=3.6cm]{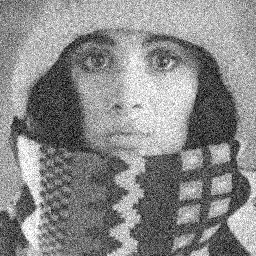}
	}
	\caption{Input images $f_\delta$ with gaussian noise of deviation $\sigma$.}
\end{figure}

\begin{figure}[H]
	\centering
	\subfloat[Mask L2-ADJ-T.]{
		\includegraphics[width=3.6cm]{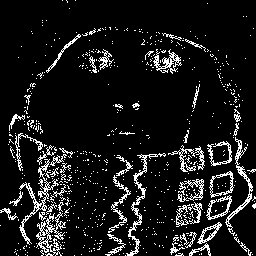}
	}
	\quad
	\subfloat[Reconstruction with L2-ADJ-T.]{
		\includegraphics[width=3.6cm]{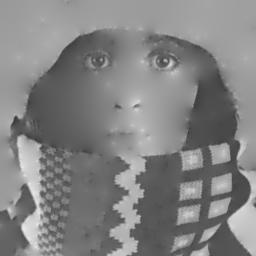}
	}
	\quad
	\subfloat[Mask L2-ADJ-H.]{
		\includegraphics[width=3.6cm]{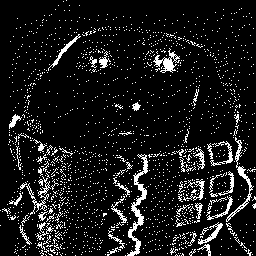}
	}
	\quad
	\subfloat[Reconstruction with L2-ADJ-H.]{
		\includegraphics[width=3.6cm]{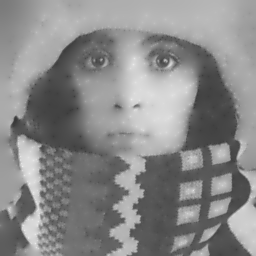}
	}
\end{figure}
\begin{figure}[H]
	\centering
	\subfloat[Mask H1-T.]{
		\includegraphics[width=3.6cm]{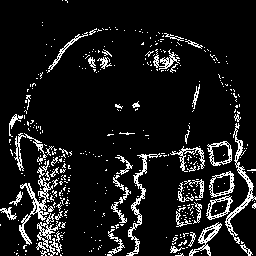}
	}
	\quad
	\subfloat[Reconstruction with H1-T.]{
		\includegraphics[width=3.6cm]{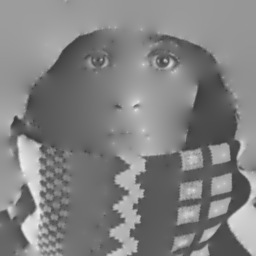}
	}
	\quad
	\subfloat[Mask H1-H.]{
		\includegraphics[width=3.6cm]{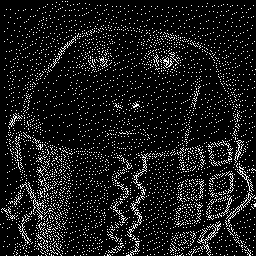}
	}
	\quad
	\subfloat[Reconstruction with H1-H.]{
		\includegraphics[width=3.6cm]{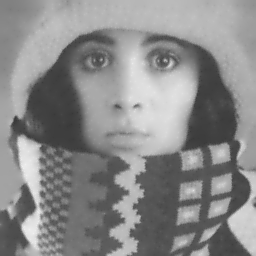}
	}
	\caption{Masks and reconstructions from image without noise and with $10\%$ of total pixels saved.}
	
	\label{fig:w:0:l2:0.1}
\end{figure}

\begin{figure}[H]
	\centering
	\subfloat[Mask L2-ADJ-T.]{
		\includegraphics[width=3.6cm]{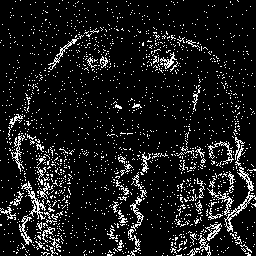}
	}
	\quad
	\subfloat[Reconstruction with L2-ADJ-T.]{
		\includegraphics[width=3.6cm]{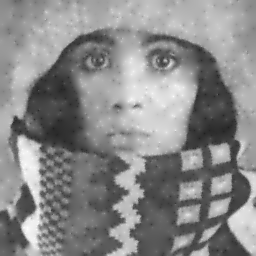}
	}
	\quad
	\subfloat[Mask L2-ADJ-H.]{
		\includegraphics[width=3.6cm]{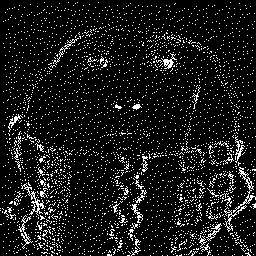}
	}
	\quad
	\subfloat[Reconstruction with L2-ADJ-H.]{
		\includegraphics[width=3.6cm]{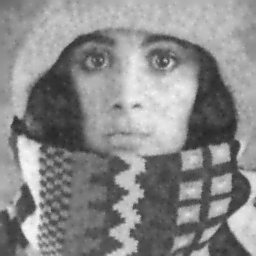}
	}
\end{figure}
\begin{figure}[H]
	\centering
	\subfloat[Mask H1-T.]{
		\includegraphics[width=3.6cm]{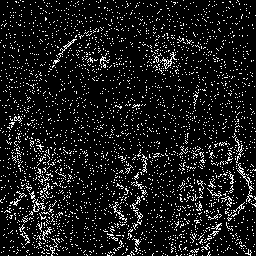}
	}
	\quad
	\subfloat[Reconstruction with H1-T.]{
		\includegraphics[width=3.6cm]{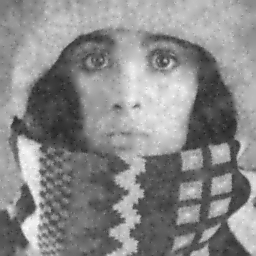}
	}
	\quad
	\subfloat[Mask H1-H.]{
		\includegraphics[width=3.6cm]{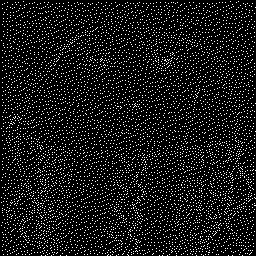}
	}
	\quad
	\subfloat[Reconstruction with H1-H.]{
		\includegraphics[width=3.6cm]{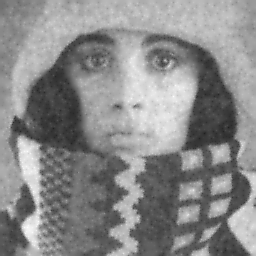}
	}
	\caption{Masks and reconstructions from image with gaussian noise of deviation $\sigma=0.03$ and with $10\%$ of total pixels saved.}
	\label{fig:w:0.03:l2:0.1}
\end{figure}

\begin{figure}[H]
	\centering
	\subfloat[Mask L2-ADJ-T.]{
		\includegraphics[width=3.6cm]{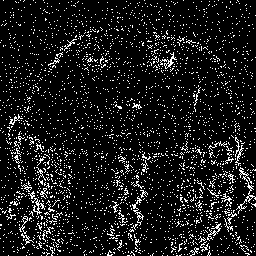}
	}
	\quad
	\subfloat[Reconstruction with L2-ADJ-T.]{
		\includegraphics[width=3.6cm]{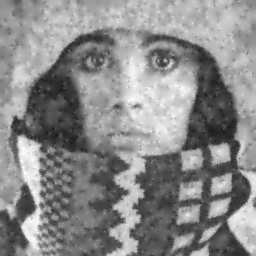}
	}
	\quad
	\subfloat[Mask L2-ADJ-H.]{
		\includegraphics[width=3.6cm]{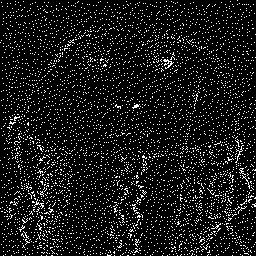}
	}
	\quad
	\subfloat[Reconstruction with L2-ADJ-H.]{
		\includegraphics[width=3.6cm]{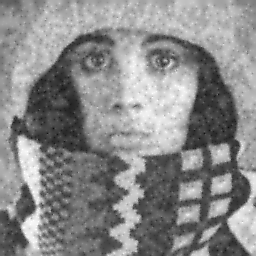}
	}
\end{figure}
\begin{figure}[H]
	\centering
	\subfloat[Mask H1-T.]{
		\includegraphics[width=3.6cm]{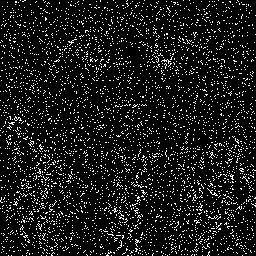}
	}
	\quad
	\subfloat[Reconstruction with H1-T.]{
		\includegraphics[width=3.6cm]{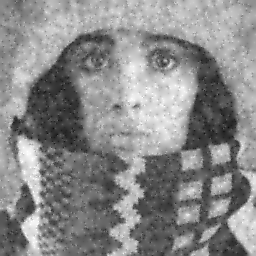}
	}
	\quad
	\subfloat[Mask H1-H.]{
		\includegraphics[width=3.6cm]{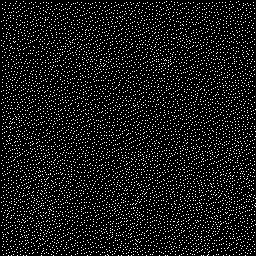}
	}
	\quad
	\subfloat[Reconstruction with H1-H.]{
		\includegraphics[width=3.6cm]{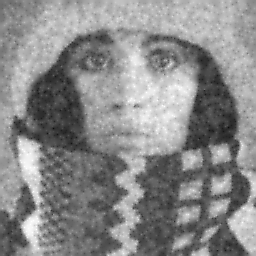}
	}
	\caption{Masks and reconstructions from image with gaussian noise of deviation $\sigma=0.1$ and with $10\%$ of total pixels saved.}
	\label{fig:w:0.05:l2:0.1}
\end{figure}

\section*{Conclusion and Discussions}

In this article, we have formulated the PDE-based compression problem as a shape optimization one, and we have performed the topological expansion for the optimality condition by the adjoint method. Following the approaches of \cite{Amstutz2006} and \cite{Garreau2001}, we compute the asymptotic development for variations of the cost functionals with general exponents $p> 1$, which leads to an analytic soft threshold criterion to select relevant pixels of the mask. The inpainting from the masks to reconstruct the images is performed with a Laplacian, but all the approach may be extended without significant changes to more involved linear operator of second order. Moreover, it can be extended to other (linear and non linear) elliptic operators, and other form of insertions (not necessarily discs) at the price of some technicalities and computations details. 
Finally, we presented some numerical experiments in the case of the $L^2$-error and of a regularized $L^1$-error. It appears that this method for selecting the mask  outperforms the other expansions when the image to compress contains gaussian noise or impulse noise and is easy to implement with a reasonable cost, the adjoint problem is linear even if the operator for the reconstruction is nonlinear.

\bibliographystyle{plainurl}
\bibliography{main}




\end{document}